\documentclass[a4paper,8pt]{article}
\usepackage[english]{babel}
\usepackage{geometry}
\usepackage[T1]{fontenc}
\usepackage[latin1]{inputenc}
\usepackage{amsmath,amssymb,mathrsfs,amsthm}
\usepackage{soul}
\usepackage{epsfig}
\usepackage{hyperref}
\usepackage{graphicx}
\usepackage{datetime}
\usepackage{fancyhdr}
\fancypagestyle{plain}{
\lhead{}
\rhead{}
\lfoot{Built \today{} at \currenttime}

}

\usepackage{tocloft}

\newtheorem{theorem}{Theorem}[section]
\newtheorem{proposition}[theorem]{Proposition}
\newtheorem{claim}[theorem]{Claim}
\newtheorem{lemma}[theorem]{Lemma}

\newtheorem{corollary}[theorem]{Corollary}
\newtheorem{example}[theorem]{Example}
\newtheorem{remarque}[theorem]{Remark}

\newcommand{\p}{\mathbb{P}}

\title{On the invariant spectrum on $\p^1$}
\date{\today, \currenttime}
\author{Mounir Hajli\footnote{\small National Center for Theoretical Sciences (Taipei Office),
 National Taiwan University,
Taipei 106, Taiwan \newline \emph{E-mail}: \ttfamily{hajlimounir@gmail.com}}}
\begin{document}

\maketitle

\begin{abstract}
Motivated by the work of Abreu and Freitas \cite{Abreu2}, we study the invariant spectrum of the
Laplace operator associated to hermitian line bundles endowed with invariant
metrics over $\p^1$.
\end{abstract}

\section{Introduction}

Let $\mathbf{P}^1$ be the complex projective line and $\omega$ a smooth and normalized
 K\"ahler form on $\mathbf{P}^1$. We denote by $\lambda_1(\omega)$
the first eigenvalue of the Laplace operator defined by $\omega$ and acting on  smooth functions on $\mathbf{P}^1$. In
\cite{Hersch}, Hersch showed that
\[
\lambda_1(\omega)\leq 2.
\]
In \cite{Abreu2}, Abreu and Freitas studied the invariant spectrum of invariant metrics on $\mathbf{P}^1$. Their goal
was to analyze this type of inequality in  the invariant setting. We denote by $0=\lambda_0(\omega)<\lambda_1(\omega)<\ldots$
the invariant eigenvalues of the Laplace operator defined by $\omega$. Their first result  shows
that there is no general analogue of Hersch's theorem (\cite[theorem 1]{Abreu2}). Nevertheless, when
they consider the class of invariant metrics that are isometric to a surface of revolution in $\mathbf{R}^3$, they
gave an optimal upper bounds for the invariant eigenvalues associated to this class. Their second result is the
following theorem
\begin{theorem}(\cite[Theorem 2]{Abreu2})\label{y3} Within the class of smooth,  invariant and normalized
 K\"ahler form $\omega$ on $\mathbf{P}^1$ and corresponding to a surface of revolution in $\mathbf{R}^3$, we have
 \[
 \lambda_j(\omega)<\frac{\xi^2_j}{2}\quad   j=1,2,\ldots
  \]
  where $\xi_j$ is the $\frac{1}{2}(j+1)$th positive zero of the Bessel function $J_0$ if $j$ is odd, and
the  $\frac{1}{2}j$\,th positive zero of $J_0'$ if $j$ is even. These bounds are optimal.
\end{theorem}

In \cite{Colbois}, Colbois, Dryden and El Soufi considered a more 
general situation. Namely, the subsequence 
$\lambda_j^G$ of the spectrum of a Riemannian manifold  $M$ which corresponds
to metrics and functions invariant under the action of a compact Lie 
group $G$. If dimension of $G$ is at least $1$, they showed that
the functional $\lambda_j^G$ admits no extremal metric under 
volume-preserving $G$-invariant deformations, cf. 
\cite[Theorem 1.1]{Colbois}. If, moreover,  $M$ has dimension at least
$3$, they proved that the functional $\lambda_j^G$ is unbounded when
restricted to any conformal class of $G$-invariant metrics of fixed
volume, cf. \cite[Theorem 1.2]{Colbois}. When $M=\mathbf{S}^n$ is 
equipped with the standard $O(n)$-action and we require that
the metric to be induced by an embedding of $\mathbf{S}^n$ in
$\mathbf{R}^{n+1}$, they gave an optimal upper bound on
$\lambda_j^G$, cf. \cite[Theorem 1.7]{Colbois}. In particular, this
result generalizes Theorem \ref{y3}.\\

Let $\omega$ be a smooth  and normalized volume form on $\mathbf{P}^1$ and $h$ a smooth  hermitian metric
on  the holomorphic line bundle $\mathcal{O}(m)$ over $\mathbf{P}^1$ 
($m\in \mathbf{N}$). We assume that $\omega$ and $h$ are invariant
under the standard action of $\mathbf{S}^1$.
The  goal of this paper is the study of the invariant spectrum of the Laplace operator $\Delta_{\omega,h}$
defined by $\omega$ and $h$, and acting on the space of smooth functions with coefficients in $\mathcal{O}(m)$.
We denote by $\lambda_0(\omega,h)=0<\lambda_1(\omega,h)<\lambda_2(\omega,h)<\ldots$ the invariant eigenvalues of $\Delta_{\omega,h}$.
Using symplectic coordinates, we attach to $\omega$ (resp. $h$) a continuous function
$g_\omega$ (resp. $h_\omega$) defined on $[0,1]$. Then our main result
is the following theorem

 \begin{theorem}\label{z1}(see Theorem \ref{y4}) Let $m\in \mathbf{N}$. Let $\omega$ and $h$ be 
 as before.
We suppose that $1<\frac{\overline{h}_\omega(x)}{\min(1,2(1-x))^m}
$ 

$<\frac{2\min(x,1-x)}{\overline{g}_\omega(x)},$ for any 
$x\in [0,1]$. Then

 \[
 \lambda_j(\omega,h)< \frac{\xi_{m,j}^2}{2}, \quad  j=1,2,\ldots,
 \]
where $\xi_{m,j}$ is a zero of the function $\frac{d}{dz}(z^{-m}J_0J_{m})$ such that $0<
\frac{\xi_{m,1}^2}{2}
<\frac{\xi_{m,2}^2}{2}<\ldots $ and
 $J_m$ is the Bessel function of order $m$.
Moreover, these bounds are optimal.
 \end{theorem}

To prove the upper bound for $\lambda_j(\omega)$  in 
\cite[Theorem 2]{Abreu2}, the authors used the Rayleigh quotient
and showed that the supremum is attained by  the union of two 
disks of equal area, a singular surface. This surface corresponds to a 
singular metric $g_{\max}$ (see \cite[p. 225]{Abreu2}). Notice  also that a similar approach 
is used 
in the proof of Theorem 1.7 in  \cite{Colbois}.\\

Let 
   $\overline{g}_{can}$ be the function on $[0,1]$ given by   $\overline{g}_{can}(x):=2\min(x,1-x)
   ,$ for any $
   x\in [0,1]$. This function corresponds to  
   $\overline{g}_{\max}$ considered in \cite[p. 225]{Abreu2} which played 
   an important role, 
   and 
    the  metrics  in Theorem 
 \ref{y3} correspond to a set of  smooth  functions $\overline{g}_\omega$ on
 $[0,1]$ such that $0\leq \overline{g}_\omega<\overline{g}_{can}$, see Remark 
 \ref{z4}.

 Comparing to \cite{Abreu2} and \cite{Colbois}, our approach for the
proof of Theorem \ref{z1} is different, for instance 
the upper bounds correspond to the invariant eigenvalues 
of a singular Laplacian associated to some 
 metrics defined uniquely in terms of
the geometry of $\mathbf{P}^1$. First, we 
show that $\overline{g}_{can}$ (equivalently $\overline{g}_{\max}$) defines 
 a singular volume form $\omega_{can}$ on  $\mathbf{P}^1$, and
 that $\overline{h}_{m,\infty}(x):=\max(1,2(1-x))^m$ corresponds to
 a singular metric on $\mathcal{O}(m)$ denoted  by 
 $\|\cdot\|_{m,\infty}$. This is done by extending the formalism
  of symplectic
  coordinates to a large class of singular metrics.  The  proof of Theorem \ref{z1}  is  mainly 
  based on the theory of the singular Laplacians
  associated to canonical metrics on $\mathbf{P}^1$ developped
   in \cite{Mounir1}.  We show that the proof  follows easily from the
  explicit computation of the spectrum of the singular Laplacian defined by 
  $\omega_{can}$ and $h_{m,\infty}$. 
  In particular, the upper bounds in
  Theorem \ref{y3} (that is the case $m=0$ in Theorem \ref{z1}) are in fact
   the invariant eigenvalues $\lambda_j(\omega_{can})$ 
  of the singular Laplacian associated to   $
   \mathbf{P}^1$ endowed with $\omega_{can}$. 
As a consequence, we recover  the previous result of Abreu and Feitas. When
$m\geq 1$, the upper bound $\frac{\xi^2_{m,j}}{2}$ is the $j$ th invariant
eigenvalue of the singular Laplacian associated to $\mathbf{P}^1$ endowed
with $\omega_{can}$ and $\mathcal{O}(m)$ equipped with the metric
$\|\cdot\|_{m,\infty}$. \\

In Paragraph \eqref{www1}, we construct 
$(\omega_\mu)_{\mu\geq 0}$ (resp. $(h_{\omega_\mu})_{\mu\geq 0}$) 
a sequence of smooth normalized volume forms (resp. smooth metrics) on $\mathbf{P}^1$
(resp. on $\mathcal{O}(m)$) such that 
\[
\lim_{\mu\rightarrow \infty}\sup_{x\in [0,1]}\frac{\overline{g}_{\omega_\mu}\overline{h}_{\omega_\mu}}{
\overline{g}_{can}\overline{h}_{\infty,m}}=\infty,\quad\text{and}\quad
\lim_{\mu \rightarrow \infty}\lambda_1(\omega_\mu,h_{\omega_{\mu}})=\infty.
\]

 As mentioned  above, the functions $\overline{g}_{can}$ 
 and $\overline{h}_{m,\infty}$ are associated to some 
 singular metrics on line bundles on $\mathbf{P}^1$. These 
 metrics are called \textit{the canonical 
  metrics} on $\mathbf{P}^1$, and they play a crucial role
  in this article. More generally, if $X$ is  a
  projective toric manifold of dimension $n$, $T(\simeq (\mathbf{C}^\ast)^{n})$ is its associated torus and $L$ a 
  $T$-equivariant line bundle
  on $X$, then we can attach canonically a continuous hermitian 
   metric
  $\| \cdot\|_{L,\infty}$ to $L$. The metric $\|\cdot\|_{L,
  \infty
  }$   is  defined uniquely in terms of  the
  combinatorial structure of $X$ and is called \textit{the canonical 
  metric} of $L$.  Let $p\in \mathbf{N}_{\geq 2}$, we denote by  $[p]:X\rightarrow X$  the 
  morphism extending the following morphism $T\rightarrow T,
  (z_1,\ldots,z_n)\mapsto (z_1^p,\ldots,z_n
  ^p)$. Since $L$ is a $T$-equivariant line bundle, we have a natural isomorphism $\theta: L^{\otimes p}\simeq 
  [p]^\ast L $. We can show that $\|\cdot\|_{L,\infty}$ is the 
  unique metric on $L$ such that $\theta$ induces an isometry of
  continuous hermitian line bundles.  Assume
   that $L$ is 
  generated by its global sections, then $L$ defines $\phi_L$, 
  an 
  equivariant embedding   into a projective space $\mathbf{P}^N$, and we can prove that $\|\cdot\|_{L,\infty}=\phi_L^\ast 
  \|\cdot\|_{{}_{\mathcal{O}(1),\infty}}$, where 
  $\|\cdot\|_{{}_{\mathcal{O}(1),\infty}}$ is the canonical
  metric of $\mathcal{O}(1)$. 
One can check, for instance \cite[Paragraph 3.3]{Maillot}
for more details.

%%%%%%%%%%%%%%%%%%%%%%%%%%%%%%

\section{Preliminary constructions}\label{p1}

 We present here  a slight generalization of the
symplectic coordinates formalism to a large class of singular K\"ahler metrics on $\mathbf{P}^1$. This will allow us in
particular
to associate a continuous K\"ahler form $\omega_{can}$ to $\overline{g}_{can}$.

We denote by $\mathcal{G}$ the set of continuous functions $\overline{g}$ on $[0,1]$, positive on $]0,1[$ and
 verifying $\overline{g}=\overline{g}_{can}
+O(\overline{g}_{can})$ near the boundary of $[0,1]$.   Using the following transformation  $x\rightarrow \frac{1}{2}\overline{g}(2x-1)$,
one can see that the functions   $\overline{g}$ satisfying the conditions of 
Theorem \ref{y3} belong to $\mathcal{G}$ and $\overline{g}_{can}$ corresponds to $\overline{g}_{\max}$ via this transformation. We prove the following result
\begin{theorem}(see Theorem \ref{z2})
For any $\overline{g}\in \mathcal{G}$, there exists a continuous, normalized and invariant
volume form $\omega$ such that $\overline{g}_\omega=\overline{g}$.

\end{theorem}

 Let $\omega$ be a continuous  and invariant volume form  on $\mathbf{P}^1$ such that $\int_{\mathbf{P}^1}\omega =1$.
 There
exists $\|\cdot\|_\omega$, a  hermitian and invariant metric of class $\mathcal{C}^2$ on $\mathcal{O}(1)$ such that
$\omega=c_1(\mathcal{O}(1),\|\cdot\|_\omega)$. We denote by $\Psi_\omega$ the function on $\mathbf{C}$ given by
\[\Psi_\omega(z):=-\frac{i}{2\pi}\log\|1\|_m(z),\quad  z\in \mathbf{C},\]
and we set
$F_\omega(u):=\log \|1\|_\omega(\exp(-u))$ for any  $u\in \mathbf{R}$, where $1$ corresponds to the global section $x_0^m$
 and $\exp(-\cdot)$ is the following
 application $\mathbf{R}\rightarrow \mathbf{P}^1, u\rightarrow [1:e^{-u}]$. Recall that we have a diffeomorphism
$ \mathbf{R}\times \frac{\mathbf{R}}{2\pi \mathbf{Z}}\simeq \mathbf{C}^\ast$, given by $(u,\theta)\rightarrow e^{-u}e^{i\theta}$. Then on $\mathbf{C}^\ast$ we have
\begin{equation}\label{x10}
\omega_{|_{\mathbf{C}^\ast}}=-\frac{1}{2\pi} \frac{\partial^2 F_\omega}{\partial u^2}(u)du\wedge d\theta\quad\text{and}\quad \frac{\partial^2 \Psi_\omega}{\partial z\partial\overline{z}}(z)=\frac{1}{2}e^{2u}\frac{\partial^2F_\omega}{
\partial u^2}(u).
\end{equation}
Let $\check{F}_\omega$ be the Legendre-Fenchel transform associated to $F_\omega$ that is the function given on $\mathbf{R}$
by
\[
\check{F}_\omega(x)=\inf_{u\in \mathbf{R}}(x\cdot u-F_\omega(u)),
\]
 One  shows that $\check{F}_\omega$ is concave and  $\check{F}_\omega(x)$ is finite if and only if  $x\in [0,1]$. We claim that
  $\check{F}_\omega$ is a function of class $\mathcal{C}^2$ on $]0,1[$. Indeed, since $F_\omega$ is $\mathcal{C}^2$
   then the function $\theta:
   u\rightarrow \frac{\partial F}{\partial u}$ defines a $\mathcal{C}^1$-diffeomorphism from $\mathbf{R}$ onto its image (we will
    show that the image is necessarily  equal to $]0,1[$). Let
   $x\in ]0,1[$. Since $F_\omega$ is strictly concave then there exists a unique element $G_\omega(x)\in \mathbf{R}$ such that
   $\check{F}_\omega(x)=
 x G_\omega(x)-F_\omega(G_\omega(x))$. Moreover, since $F_\omega$ is $\mathcal{C}^2$ then
  $x=\frac{\partial F_\omega}{\partial u}(G_\omega(x))$. Thus $G_\omega$ is the inverse function of $\frac{\partial F_\omega}{\partial u}$.
   In particular we deduce that $]0,1[$ is included in the image of $\theta$. By differentiating  the previous
  identity we obtain
 $\frac{\partial \check{F}_\omega}{\partial x}(x)=G_\omega(x)$ for any $x\in ]0,1[$. Since
 $\int_{\mathbf{P}^1}\omega=1$, and  using \ref{x10} we deduce that
 $\frac{\partial F_\omega}{\partial u}(-\infty)-\frac{\partial F_\omega}{\partial u}(+\infty)=1$, thus
$\frac{\partial F_\omega}{\partial u}(+\infty)$ and $\frac{\partial F_\omega}{\partial u}(-\infty)$ are finite.
we conclude  that $\frac{\partial F_\omega}{\partial u}(+\infty)=0$ and $\frac{\partial F_\omega}{\partial u}(-\infty)=1$ and $\theta$ is
a $\mathcal{C}^1$-diffeomorphism between $\mathbf{R}$ and $]0,1[$. We  can then  consider
the following change of coordinates $x=\frac{\partial F_\omega}{\partial u}(u)$.
 We set \[g_\omega(x):=-G_\omega'(x)\quad\text{ and}\quad \overline{g}_\omega(x):=\frac{1}{g_\omega(x)} \quad x\in ]0,1[.\]
  One checks that
 $\overline{g}_\omega(x)=-\frac{\partial^2 F_\omega}{\partial u^2}
 (G(x)),$ for any  $x\in ]0,1[ $.

We set $\overline{h}_\omega$ the function on $[0,1]$ given by \[\overline{h}_\omega(x):=h(1,1)(e^{-G_\omega(x)}),\quad  x\in [0,1].\]

In Corollary
\ref{y7},
 we
 show that $\overline{h}_\omega$ is a continuous function on $[0,1]$, positive on $[0,1[$ such that $\lim_{x\rightarrow 1^-}
\frac{\overline{h}_\omega(x)}{(1-x)^m}$ exists and positive. Moreover, we  prove that any function $\overline{h}$ satisfying
the previous conditions defines a continuous and invariant hermitian metric on $\mathcal{O}(m)$.

\section{On the invariant metrics on $\mathbf{P}^1$}\label{r1}

We denote by  $[x_0:x_1]$ the homogenous
coordinate  on $\mathbf{P}^1$ and $z=x_1/x_0$ the affine coordinate over the open subset $\mathbf{C}=\{x_0\neq 0\}$. Let  $\mathbf{C}^\ast$ be
the complex  torus acting on $\mathbf{P}^1$ as a toric
manifold
and $\mathbf{S}^1$  the compact sub-torus in $\mathbf{C}^\ast$.

Let $m\in \mathbf{N}$. Let $\|\cdot\|$ be a continuous hermitian  metric on the line bundle $\mathcal{O}(m)$ over $\mathbf{P}^1$,
 and we suppose that $\|\cdot\|$ is invariant
under the action of  $\mathbf{S}^1$. To the metric $\|\cdot\|$ we associate a continuous function $F_{\|\cdot\|}$  defined on $\mathbf{R}$ as follows
\[
F_{\|\cdot\|}(u)=\log \|1\|(\exp(-u)),\quad  u\in \mathbf{R}
\]
  Let $\|\cdot\|_{m,\infty}$ be
the following continuous hermitian  metric on $\mathcal{O}(m)$ given by
\begin{equation}\label{x7}
\|s\|_{m,\infty}(z)=\frac{|s(z)|}{\max(1,|z|)^m},\quad z\in \mathbf{C}\,
\end{equation}
 where $s$ is a local holomorphic section of $\mathcal{O}(m)$. $\|\cdot\|_{m,\infty}$ is called  \textit{the canonical metric}
 of $\mathcal{O}(m)$. This metric is defined by the structure of $\mathbf{P}^1$ viewed as a toric manifold. In fact, we can show
 that any equivariant line bundle on a projective toric manifold  admits a continuous metric of the same nature
 defined uniquely by the combinatorial structure of the manifold (see 
 Paragraph 3.3 of \cite{Maillot} and references therein).
 \begin{remarque}
 Unfortunately  the spectral theory of Laplacians (see for instance \cite{heat}) does not hold for $\|\cdot\|_{m,\infty}$, since
 it is a singular metric. In \cite{Mounir8} and \cite{Mounir9}, we developed  a generalized spectral theory for a large class  of singular metrics
which includes 
the case of the  canonical metrics on $\mathbf{P}^1$.

 \end{remarque}

 We have  
 \begin{equation}\label{l1}
 F_{m,\infty}(u):=F_{\|\cdot\|_{m,\infty}}(u)=m\min(0,u),
 \quad  u\in \mathbf{R}.
 \end{equation}
   We can  establish  that
there exists a bijection between the set
of continuous  hermitian and invariant metrics on $\mathcal{O}(m)$ and the set of continuous
 functions $F$ on $\mathbf{R}$ such that
the function $\mathbf{C}^\ast\rightarrow \mathbf{R}, z\rightarrow F(-\log|z|) -F_{m,\infty}(-\log |z|)$ extends to a bounded
continuous function on $\mathbf{P}^1$ (see for instance \cite[Proposition
4.3.10]{Burgos2}).\\

 \begin{example}\label{x2}
Following \cite{Guillemin1} and
\cite{Guillemin}, the   Fubini-Study form $\omega_{FS}$ is viewed  as  the "canonical" K\"ahler
   metric on $\mathbf{P}^1$ which is compatible with  the standard moment map on
   $\mathbf{P}^1$.
   We set $F_0=F_{\omega_{FS}}$. An easy computation shows
 that for any $u\in \mathbf{R}$, we have
\[
F_{0}(u)=-\frac{1}{2} \log(1+e^{-2u}),\,
\frac{\partial F_0}{\partial u}(u)=\frac{e^{-2u}}{1+e^{-2u}},\,
\frac{\partial^2 F_{0} }{\partial u^2}(u)=-\frac{2e^{-2u}}{(1+e^{-2u})^2},\]
and for any $x\in [0,1]$
\[
 \,G_0(x)=-\frac{1}{2}\log\Bigl(\frac{x}{1-x}\Bigr),\quad 
 \overline{g}_0(x)=2x(1-x).
\]

\end{example}
\begin{example}\label{z3}
The second example is a singular volume form defined by the combinatorial structure of $\mathbf{P}^1$.
Notice that $T\mathbf{P}^1$ is isomorphic to $\mathcal{O}(2)$ then the metric $\|\cdot\|_{2,\infty}$ (see \eqref{x7}) induces a continuous
 volume form $\omega_{can}$ on $\mathbf{P}^1$. This form is given on $\mathbf{C}$ as follows
 \[
 \omega_{can}=\frac{i}{4\pi}\frac{dz\wedge d\overline{z}}{\max(1,|z|)^4}.
 \]
One checks  that $\int_{\mathbf{P}^1}\omega_{can}=1$. We consider the following hermitian
 continuous metric $\|\cdot\|_{can}$ on $\mathcal{O}(1)$  defined as follows
\[
\|s\|^2_{can}(z):=\frac{|s(z)|^2}{\max(1,|z|)^2}\exp(-k(z)),\quad z\in \mathbf{C},
\]
 where $s$ is a local holomorphic section of $\mathcal{O}(1)$ and $k(z)=\frac{1}{2}\min(|z|^2,\frac{1}{|z|^2}),$ for any $z\in \mathbf{C}$. We
have the following result
\begin{proposition}\label{p2}
The metric $\|\cdot\|_{can}$ is  positive (i.e the current $c_1(\mathcal{O}(1),\|\cdot\|_{can})$ is positive)
  and \[
  c_1(\mathcal{O}(1),\|\cdot\|_{can})=\omega_{can}.\]
\end{proposition}
 \begin{proof}
 We have the following equality of currents
 \[
 c_1(\mathcal{O}(1),\|\cdot\|_{can})=c_1(\mathcal{O}(1),\|\cdot\|_{1,\infty})+[dd^c k]
 \]
From \cite[Corollary 6.3.5]{Maillot}, we have $c_1(\mathcal{O}(1),\|\cdot\|_{1,\infty})=\delta_{\mathbf{S}^1}$  (the current
of integration on $\mathbf{S}^1$).
Let $f$ be
 a smooth function on $\mathbf{P}^1$. We have
 \begin{align*}
 [dd^ck](f)=&\int_{\mathbf{P}^1} k\,dd^cf\\
 =&\frac{1}{2}\int_{|z|\leq 1} |z|^2 dd^cf+\frac{1}{2}\int_{|z|\geq 1} |z|^{-2} dd^cf\\
 =&\frac{1}{2}\int_{|z|\leq 1} fdd^c|z|^2 +\frac{1}{2}\int_{\mathbf{S}^1}(d^c f-fd^c|z|^2)+\frac{1}{2}
 \int_{|z|\geq 1} fdd^c |z|^{-2}\\
  &-\frac{1}{2}\int_{\mathbf{S}^1} (d^c f+fd^c|z|^{-2})\quad\text{by Stockes' 
  theorem,}\\
 =&\frac{1}{2}\int_{|z|\leq 1} fdd^c|z|^2+\frac{1}{2}\int_{|z|\geq 1} fdd^c |z|^{-2}-
 \int_{\mathbf{S}^1}f d^c|z|^2.
 \end{align*}
 Therefore,
 \[
 \int_{\mathbf{P}^1}
 fc_1(\mathcal{O}(1),\|\cdot\|_{can})=\frac{i}{4\pi }\int_{|z|\leq 1} f dz\wedge d\overline{z}+\frac{i}{4\pi }\int_{|z|\geq 1} f
 \frac{dz\wedge d\overline{z}}{|z|^4}=\int_{\mathbf{P}^1}f\omega_{can}.
 \]
Which concludes the proof of the proposition.
 \end{proof}

 We denote by $F_{can}$ the function $\mathbf{R}\rightarrow \mathbf{R}$,
  $u\rightarrow \log \|\cdot\|_{can}(\exp(-u))$. We have
 $F_{can}(u)=\min (0,u)-\frac{1}{4}\min(e^{-2u},e^{2u})$ for any $u\in \mathbf{R}$. $F_{can}$ is smooth on $\mathbf{R}\setminus \{0\}$. 
 Let $x\in [0,1]$. Since $F_{can}$ is strictly concave (this is follows from
 Proposition \ref{p2}), then there exists a unique
  $u\in \mathbf{R}$ such that $\check{F}_{can}(x)=xu-F_{can}(u)$. First, we suppose that $u\neq 0$,  then $u$ satisfies 
  $x-\frac{\partial F_{can}}{\partial u}(u)=0$. It follows that
  , $x=1-\frac{1}{2}e^{2u}$ if $u<0$ and $x=\frac{1}{2}e^{-2u}$ if $u>0$. By
  continuity, we can deduce that $\check{F}_{can}$ is  given by the following expression
 \begin{equation*}
\check{F}_{can}(x) =
\begin{cases}
-\frac{1}{2}x \log (2x)+\frac{1}{2}x & \text{if }  x\in [0,1/2],\\
-\frac{1}{2}(1-x)\log (2(1-x))+\frac{1}{2}(1-x) & \text{if } x\in [1/2,1].
\end{cases}
\end{equation*}

We see that $\check{F}_{can}$ is a $\mathcal{C}^2$ function on $[0,1]$. Then we let
 $G_{can}$ be  the function on $[0,1]$ given by $G_{can}(x)=\frac{d \check{F}_{can}}{dx }(x),$ for any $ x\in
 [0,1]$. We have
 \begin{equation*}
G_{can}(x) =
\begin{cases}-\frac{1}{2}\log(2x) & \text{if }  x\in [0,1/2],\\
\frac{1}{2}\log (2(1-x)) & \text{if } x\in [1/2,1],
\end{cases}
\end{equation*}
 We notice that $G_{can}$ defines a bijection between $]0,1[$ and $\mathbf{R}$.
 We set $g_{can}=-G'_{can}$ and $\overline{g}_{can}=\frac{1}{g_{can}}$. Then
\begin{equation}\label{metcan}
\overline{g}_{can}(x)=2\min(x,1-x),\quad x\in [0,1].
\end{equation}

We set $\overline{h}_{m,\infty}$ the function on $[0,1]$ given by
  \begin{equation}\label{metcan1}
  \overline{h}_{m,\infty}(x):=\|1\|^2_{m,\infty}(e^{-G_{can}}(x))
=\min(1,2(1-x))^m.
\end{equation}
\end{example}

\begin{remarque}\label{z4}
In \cite[Section 4]{Abreu2}, Abreu and Freitas constructed a class of smooth metrics $g$. These metrics correspond to closed surfaces
of revolution in $\mathbf{R}^3$ and  they proved that  $\overline{g}:=\frac{1}{g}$ is a  smooth functions  on $[-1,1]$  satisfying
\begin{equation}\label{x8}
\overline{g}(-1)=g(1)=0, \quad \overline{g}'(-1)=2=-\overline{g}'(1)\,\text{and}\, \sup_{[-1,1]}|\overline{g}'(x)|\leq 2.
\end{equation}
Clearly $\overline{g}\leq \overline{g}_{\max}$, where $\overline{g}_{\max}(x)=2(1-|x|)$ for any $x\in [0,1]$. Let $j\in \mathbf{N}_{\geq 1}$ and $\lambda_j(g)$
the $j$-th invariant eigenvalue of the Laplace operator defined by  $g$. They showed that
$\lambda_j(g)$ viewed as a function with variable $g$ is bounded over  the set of  smooth and invariant
 metrics corresponding to surfaces of revolution, see  \cite[Theorem 2]{Abreu2}.

Using the following transformation $\overline{g}_{[0,1]}(x):=\frac{1}{2}\overline{g}(2x-1)$ for any $x\in [0,1]$ (In particular,
$\overline{g}_{can}(x)=\frac{1}{2}\overline{g}_{\max}(2x-1)$), we see that the
smooth functions $\overline{g}$ on $[-1,1]$ satisfying \eqref{x8} belong to $\mathcal{G}$, up to the previous transformation.
As we may expect the set $\mathcal{G}$ is not reduced to functions satisfying \ref{x8}. More precisely,
we  prove (see below)
that there exist functions $\overline{g}\in \mathcal{G}$ 
satisfying $\overline{g}(-1)=g(1)=0,\overline{g}'(-1)=2=-\overline{g}'(1)$ and $\overline{g}\leq
\overline{g}_{\max}$ such that $\sup_{[-1,1]}|\overline{g}'(x)|$ can be a  large real number.

 \end{remarque}

\begin{claim}
For any $A>0$, there exists a smooth function  $\overline{g}_A$ on $[0,1]$ such that $\overline{g}_A
\in \mathcal{G}$  and
$\lim_{A\rightarrow \infty}\sup_{x\in [0,1]}|\overline{g}_A(x)|=\infty$.
\end{claim}

\begin{proof}
Let $\rho$ be a non-zero, positive and smooth function on $\mathbf{R}$ with support in  $[1/4,3/4]$ and
bounded from above
 by $1/8$. Let $A\geq 1$, one checks  that $\rho(A(x-1/2)+1/2)\leq 1/2\min(x,1-x),$ for any $ x\in [0,1]$.
It follows that
$\rho(A(x-1/2)+1/2)+2x(1-x)\leq 2\min(x,1-x)$ for any $x\in [0,1]$. We set
$\overline{g}_A(x)=2x(1-x)+\rho(A(x-1/2)+1/2),$ for  $ x\in [0,1]$. Then  it is easy to see that $\overline{g}_A$ is
a smooth function on $[0,1]$ and
belongs to $\mathcal{G}$.  Let $x_0\in [0,1]$ is such that $\rho'(x_0)\neq 0$, and we set $x_A:
=1/A(x_0-1/2)+1/2$, then
$\overline{g}_A'(x_A)=2-4x_A+2A\rho'(x_0)\sim_{A\rightarrow \infty} A$. Thus $\lim_{A\rightarrow \infty}
\sup_{x\in [0,1]}|\overline{g}'_A(x)|=\infty$.

\end{proof}

In the sequel, we keep the same notations as in Section \eqref{p1}.

\begin{proposition}
Let $\omega$ be a smooth and invariant K\"ahler form on $\mathbf{P}^1$ such that $\int_{\mathbf{P}^1}\omega=1$. We have $\overline{g}_\omega\in \mathcal{G}$.
\end{proposition}

\begin{proof}
 Recall that for any $z\in \mathbf{C}^\ast$,
$\frac{\partial^2 \Psi_\omega}{\partial z\partial\overline{z}}(z)=\frac{1}{2}e^{2u}\frac{\partial^2F_\omega}{
\partial u^2}(u)=\frac{1}{2}
e^{2G_\omega(x)}\overline{g}_\omega(x)$ and $G_\omega$ is finite over $]0,1[$. Then $\overline{g}_\omega$ is positive on
$]0,1[$. Moreover, since $\lim_{x\rightarrow 0}G_\omega(x)=+\infty$, then
 $\lim_{x\rightarrow 0}e^{2G_\omega(x)}\overline{g}_\omega(x)=2\frac{\partial^2}{\partial z\partial\overline{z}}\Psi(0)=:l_\omega
$ which is  finite and non-zero. Let $\epsilon\in ]0,1/l_\omega[$, then there exists a positive real number $\eta$ such that
$(1/l-\epsilon)\leq e^{-2G_\omega(x) }g_{\omega}(x)\leq 1/l+\epsilon
 $ for any $x<\eta$.  It follows that  $(1/l-\epsilon)x\leq -\frac{1}{2}
 \int_0^x e^{-2G_\omega(x) }G'_{\omega}(x)\leq (1/l+\epsilon)x$ for any $x\leq \eta$. Therefore
 \begin{equation}\label{x6}
 (1/l_\omega-\epsilon)2x\leq e^{-2G_\omega(x)}\leq (1/l_\omega+\epsilon)2x, \quad x\leq \eta.
 \end{equation}
It follows that
\[
\lim_{x\rightarrow 0}\frac{\overline{g}_\omega(x)}{x}=\lim_{x\rightarrow 0}\frac{e^{-2G_\omega(x)}}{x}\lim_{x\rightarrow 0}
e^{2G_\omega(x)}\overline{g}_\omega(x)=2.
\]
  We claim that $\overline{g}_\omega(x) e^{-2G_\omega(x)}=l_\omega+O(x)$ for $0<x\ll 1$. Indeed, since
  $\frac{\partial^2 \Psi_\omega}{\partial z\partial\overline{z}}$ is an invariant smooth function, then $\frac{\partial^2 \Psi_\omega}{\partial z\partial\overline{z}}(z)=l_\omega+
O(|z|^2)$ in a small open neighborhood of $z=0$. So, $\frac{\partial^2 \Psi_\omega}{\partial z\partial\overline{z}}(e^{-G(x)})=\frac{l_\omega}{2}+
O(e^{-2G_\omega(x)})$ for $0<x\ll 1$. From \eqref{x6}, we deduce that $\frac{\partial^2 \Psi_\omega}{\partial
z\partial\overline{z}}(e^{-2G(x)})=\frac{l_\omega}{2}+O(x)$ for $0<x\ll 1$.

Therefore,
\[
\frac{1}{2}\frac{d e^{-2G_\omega(x)}}{dx}=g_\omega(x) e^{-2G_\omega(x)}=\frac{1}{l_\omega}+O(x).
\]
Then
\begin{equation}\label{y5}
e^{-2G_\omega(x)}=\frac{2}{l_\omega}x+O(x^2)
\end{equation}
So,
\[
\overline{g}_{\omega}(x)=e^{-2G_\omega(x)}(l_\omega+O(x))=2x+O(x^2).
\]

To conclude the proof of the theorem we need to prove the following
\begin{equation}\label{x9}
\overline{g}_{\omega}(x)=2(1-x)+O((1-x)^2) \quad \forall\,0<1-x\ll 1.
\end{equation}
We consider the following biholomorphic map $\tau:\mathbf{P}^1\rightarrow \mathbf{P}^1, z\rightarrow
z^{-1}$. Then $\tau^\ast \omega$ is smooth, K\"ahler   and invariant. We claim that
\[
F_{\tau^\ast \omega}(-u)=-u+F_\omega(u)\quad \forall\, u\in \mathbf{R}.
\]
This is follows from  the following  equality over $\mathbf{C}^\ast$  $\|x_1\|=|z|\|x_0\|$. Then, for any $x\in [0,1]$
\[
\check{F}_{\tau^\ast\omega}(x)=\inf_{u\in \mathbf{R}}(xu-F_{\tau^\ast\omega}(u))=\inf_{u\in \mathbf{R}}(u(1-x)-F_\omega(u))=\check{F}_\omega(1-x),
\]
Thus \begin{equation}\label{y6}
G_{\tau^\ast \omega}(x)=-G_\omega(1-x),
\end{equation}  So $\overline{g}_{\tau^\ast\omega}(x)=\overline{g}_\omega(1-x)$.  We conclude that the proof of
 \eqref{x9} can be deduced  from the first part of the proof.
\end{proof}

\begin{theorem}\label{z2}
For any $\overline{g}\in \mathcal{G}$, there exists a continuous, normalized and invariant
volume form $\omega$ such that $\overline{g}_\omega=\overline{g}$.

\end{theorem}

\begin{proof}

Let $\overline{g}\in \mathcal{G}$ and we  set $g:=1/\overline{g}$. By hypothesis we can find two positive constants   $k$ and $k'$
such that
\begin{equation}\label{x5}
k\leq \frac{\overline{g}(x)}{\overline{g}_{can}(x)}=\frac{g_{can}(x)}{g(x)}\leq k',\quad  x\in [0,1].
\end{equation}
We  set \[G_g(x):=-\int_{1/2}^x g(s)ds \quad \text{and}\quad L_g(x):=\int_{1/2}^x G_g(s)ds,\quad  x\in ]0,1[
\] Since $g$ is
positive, then $L_g$ is strictly concave on $]0,1[$. By \eqref{x5} we can show that $L_g$ is of Legendre type\footnote
{ Let $C\subset \mathbf{R}$ an open convex set. A differentiable concave function $f:C\rightarrow \mathbf{R}$ is of Legendre type
 if it is strictly concave and $\lim_{i\rightarrow \infty}|\frac{\partial f}{\partial u}(u_i)|=\infty$ for every sequence
 $(u_i)_{i\geq 1}$ converging to a point in the boundary of $C$.}
 on $]0,1[$. It follows
that the function  $x\rightarrow \frac{\partial L_g}{\partial x}$ defines a $\mathcal{C}^1$-diffeomorphism  between $]0,1[$ and $\mathbf{R}$.
 Moreover, we can prove there exist  two constants $\alpha$, $\alpha'$ such that
 $\alpha\leq L_g(x)\leq \alpha',$ for any $ x\in [0,1]$. By the same arguments as in Section \eqref{p1}, we show that
   the function $F_g$ given on $\mathbf{R}$ by $F_g(u)=
 \inf_{x\in [0,1]}(ux-L_g(x))$ is of class $\mathcal{C}^2$ 
 and 
  $\frac{\partial F_g}{\partial u}$ is the inverse function of $
 \frac{\partial L_g}{\partial x}$ and
 satisfies
    \begin{equation}\label{y1}
    -\alpha'+F_{1,\infty}(u)\leq F_g(u)\leq -\alpha+F_{1,\infty}(u),\quad
    u\in \mathbf{R}.
    \end{equation}
(For the definition of $F_{1,\infty}$ see \eqref{l1}).

  We consider the following differential form on $\mathbf{C}^\ast(\simeq \mathbf{R}\times \mathbf{R}/2\pi \mathbf{Z})$
  \[
 \omega_g:= -\frac{i}{4\pi }\frac{\partial^2 F_g}{\partial u^2}e^{2u} dz\wedge d\overline{z}.
  \]

Since $\frac{\partial^2 F_g}{\partial u^2}(u)=\overline{g}(x)$ then $\omega_g$ is  positive on $\mathbf{C}^\ast$.  By definition of $\overline{g}$, we
have
 $\overline{g}(x)=2x+O(x^2)$ for $0<x\ll 1$. Then $-\frac{\partial^2F_g}{\partial u^2}(u)=2\frac{\partial F_g}{\partial u}+O((\frac{\partial F_g}{\partial u})^2
  )$ for $u\gg 1$. This  gives two equalities
  \[
  \frac{\partial}{\partial u}\bigl(e^{2u}\frac{\partial F_g}{\partial u}\bigr)^{-1}=O(e^{-2u})\,\text{and}\,\, \frac{\partial}{\partial u}\log |e^{2u}
  \frac{\partial F_g}{\partial u}|=O(
  \frac{\partial F_g}{\partial u}),\quad  u\gg 1.
  \]
  The first equality gives $(e^{2u}\frac{\partial F_g}{\partial u}(u))^{-1}-(e^{2v}\frac{\partial F_g}{\partial u}(v))^{-1}=O
  (e^{-2u}-e^{-2v}) $ for $u,v\gg 1$. This shows that the following limit $l_g:=-\lim_{u\rightarrow \infty}(
  e^{2u}\frac{\partial F_g}{\partial u}(u))^{-1}$ exists and finite. The limit $l_g$ is necessarily non-zero. Indeed, by  the second
  equality we have
 $e^{2u}\frac{\partial F_g}{\partial u}(u)\leq e^{2v}\frac{\partial F_g}{\partial u}(v)e^{O( F_g(v)-F_g(u))}$ for $u,v\gg 1$ and
 from \eqref{y1}, the RHS of the previous inequality is bounded for fixed $v$ and $u\gg 1$. Therefore,
  \[
  -e^{2u}\frac{\partial^2F_g}{\partial u^2}(u)=\frac{2}{l_g}+o(1),\quad  u\gg 1.
  \]
  Then the form $\omega_g$ extends to $ \mathbf{C}$.\\

  Let $\overline{g}^\ast$ be the function on $[0,1]$ given by 
  \[\overline{g}^\ast(x)=\overline{g}(1-x),\quad x\in [0,1].\] Clearly $\overline{g}^\ast\in \mathcal{G}$. We set
  $g^\ast=1/\overline{g}^{\ast}$. We have  $L_g(x)=L_{g^\ast}(1-x)$ for any $x\in [0,1]$.  Then
  \[
    -e^{2u}\frac{\partial^2F_{g^\ast}}{\partial u^2}(u)=\frac{2}{l_{g^\ast}}+o(1),\quad  \, u\gg 1.
  \]
  As before  we can show that
  $F_{g^\ast}(-u)=-u+F_g(u),$  for any $ u\in \mathbf{R}$.  We deduce that
  \[
    -e^{-2u}\frac{\partial^2F_{g}}{\partial u^2}(u)=\frac{2}{l_{g^\ast}}+o(1)\quad  \, (-u)\gg 1.
  \]
We conclude that $\omega_g$  extends to a positive, invariant and continuous $(1,1)$-form on $\mathbf{P}^1$.
We denote it also by $\omega_g$. Finally,
notice
that $\int_{\mathbf{P}^1}\omega_g=\frac{\partial F_g}{\partial u}(-\infty)-\frac{\partial F_g}{\partial u}(+\infty)=1$.
\end{proof}

\begin{corollary}\label{y7} Let $h$ be a continuous and invariant hermitian metric on $\mathcal{O}(m)$. Then the
function $\overline{h}_\omega$ on $[0,1]$ given by
\[
\overline{h}_\omega(x)=h(1,1)(\exp(-G_\omega(x))),\quad x\in [0,1],
\]
is continuous on $[0,1]$, positive on $[0,1[$ and the limit $\lim_{x\rightarrow 1^-}\frac{\overline{h}_\omega(x)}{(1-x)^m}$ exists, finite and non-zero. Moreover, any continuous function $\overline{h}$ verifying the previous conditions, defines
a continuous and invariant hermitian metric on $\mathcal{O}(m)$.
\end{corollary}

\begin{proof}
Let $h$ be a continuous and invariant hermitian metric on $\mathcal{O}(m)$. There exists a continuous and
invariant function $f$ on $\mathbf{P}^1$ such that $h=e^f h_{m,\infty}$. Then it suffices to prove the corollary
for the metric $h_{m,\infty}$. We have $\overline{h}_{m,\infty}(x)=\min(1,e^{2m G_\omega(x)})$ for any $x\in [0,1]$. Clearly
$\overline{h}_{m,\infty}$ is continuous on $]0,1[$. By \eqref{y5} and \eqref{y6}, we deduce that $\overline{h}_{m,\infty}(x)=1$
for $0<x\ll 1$ and $\overline{h}_{m,\infty}(x)=\frac{2^m}{l_\omega^m}(1-x)^m+O((1-x)^{m+1})$.

Now, let $\overline{h}$ be a continuous function on $[0,1]$, positive on $[0,1[$ and such that the limit
 $\lim_{x\rightarrow 1^-}\frac{\overline{h}_\omega(x)}{(1-x)^m}$ exists, finite and non-zero. The function $x\rightarrow
 \frac{\overline{h}(x)}{
\overline{h}_{m,\infty}(x)}$ is continuous and positive on $[0,1]$. Thus it extends to a continuous, positive
and invariant function on $\mathbf{P}^1$. Therefore, $\overline{h}$ defines a continuous hermitian metric on $\mathcal{O}
(m)$.

\end{proof}

\section{The invariant spectrum of the Laplacian}\label{ss1}

Let $m\in \mathbf{N}$. Let  $A^{(0,0)}(\mathbf{P}^1,\mathcal{O}(m))$ be 
the space of smooth functions on $\mathbf{P}^1$ with coefficients
in $\mathcal{O}(m)$. Let $\omega$ be a smooth,   invariant   and normalized K\"ahler form on $\mathbf{P}^1$ and $h$ an  invariant
smooth hermitian metric on $\mathcal{O}(m)$. The metrics $\omega$ and $h$ induce a $L^2$-scalar
  product $(,)_{\omega,h}$ on
 $A^{(0,0)}(\mathbf{P}^1,\mathcal{O}(m))$ given  as follows
 \[
 (s,t)_{\omega,h}=\int_{x\in \mathbf{P}^1}h(s(x),t(x))\omega(x),
 \]
 for $s,t \in A^{(0,0)}(\mathbf{P}^1,\mathcal{O}(m))$. The Cauchy-Riemann operator $\overline{\partial}_{\mathcal{O}(m)}:
 A^{(0,0)}(\mathbf{P}^1,\mathcal{O}(m))\rightarrow A^{(0,1)}(\mathbf{P}^1,\mathcal{O}(m))$ has  an adjoint for the $L^2$-scalar
  product, i.e there is a map $\overline{\partial}^\ast_{\mathcal{O}(m)}:A^{(0,1)}(\mathbf{P}^1,\mathcal{O}(m))\rightarrow
  A^{(0,0)}(\mathbf{P}^1,\mathcal{O}(m)) $  such that
  $(s,\overline{\partial}^\ast_{\mathcal{O}(m)}t)_{\omega,h}=(\overline{\partial}_{\mathcal{O}(m)}s,t)_{\omega,h}
  $ for any $s\in A^{(0,0)}(\mathbf{P}^1,\mathcal{O}(m))$ and $t\in A^{(0,1)}(\mathbf{P}^1,\mathcal{O}(m))$. The operator
  $
  \Delta_{\omega,h}:= \overline{\partial}^\ast_{\mathcal{O}(m)} \overline{\partial}_{\mathcal{O}(m)}$ acting on
  $A^{(0,0)}(\mathbf{P}^1,\mathcal{O}(m))$ is called the Laplacian operator.
  We denote by $H_2$ the completion of $A^{(0,0)}(\mathbf{P}^1,\mathcal{O}(m))$
 with respect to the norm $\|\cdot\|_2$ defined as follows
 \[
 \|s\|^2_2=\int_{\mathbf{P}^1}h(s,s)\omega+ \frac{i}{2\pi}\int_{\mathbf{P}^1}h(\frac{\partial s}{\partial \overline{z}},\frac{\partial s}{\partial \overline{z}})dz\wedge d\overline{z},
 \]
 for any $s\in A^{(0,0)}(\mathbf{P}^1,\mathcal{O}(m))$.
 It is well known
 that  $\Delta_{\omega,h}$ admits a maximal and positive self adjoint extension to $H_2$ and  has a discrete, infinite and positive spectrum. \\

  In \cite{Mounir1}, we associated to the metrics $\omega_{can}$ and $h_{m,\infty}$ a singular Laplacian operator
 $\Delta_{\overline{\mathcal{O}(m)}_\infty}$ which extends the definition of the classical one,
  and we showed that this operator has the same properties as in the classical situation. More precisely,
  we established that $\Delta_{{\overline{\mathcal{O}(m)}}_\infty}$ admits a maximal positive and self-adjoint  extension to $H_2$ see \cite[Theorem 1.3]{Mounir1} and  has a discrete, infinite and positive spectrum
  \cite[Theorem 1.4]{Mounir1}. Moreover,  we  computed
 it explicitly.
 \begin{remarque}
Following the notations in this article,  we  set
 $\Delta_{\omega_{can},h_{m,\infty}}:=2\Delta_{\overline{\mathcal{O}(m)}_\infty}$. The factor 2 is added here
  since the volume form 
 was not normalized in  \cite{Mounir1}.
 \end{remarque}

 Let  $n\in \mathbf{Z}$ and $J_n$ the Bessel function of order $n$.  We consider the  function $L_{m,n}$
defined on  $\mathbf{C}^\ast$ as follows:
\[
L_{m,n}(z)=-z^m\frac{d}{dz}\bigl(z^{-m} J_n(z)J_{n-m}(z)\bigr),\quad 
z\in \mathbf{C}^\ast.
\]
   We have
  \begin{theorem}\label{y2}
For any $m\in\mathbf{N}$, $\Delta_{\omega_{can},h_{m,\infty}}$ admits a discrete, positive and infinite  spectrum, and

{{} \[
\mathrm{Spec}(\Delta_{\omega_{can},h_{m,\infty}})=\Bigl\{0\Bigr\}\bigcup \Bigl\{\frac{\lambda^2}{2} \Bigl|\, \exists n\in \mathbf{N},\,
L_{m,n}(\lambda)=0 \Bigr\}.
\]}
If we denote by $0<\lambda_1(\omega_{can},h_{m,\infty})<\lambda_2(\omega_{can},h_{m,\infty})<\ldots $
the   invariant eigenvalues of $\Delta_{\omega_{can},h_{m,\infty}}$. Then
 the set of invariant eigenvalues of
$\Delta_{\omega_{can},h_{m,\infty}}$ is equal to $\bigl\{\frac{\lambda^2}{2} |\,  L_{m,0}(\lambda)=0  \bigr\}$.
\end{theorem}
\begin{proof}
See \cite[Theorem 1.4]{Mounir1}.
\end{proof}

Since $\mathcal{O}(m)$ is endowed with its global sections, 
then an element $\xi \in
A^{(0,0)}(\mathbf{P}^1,\mathcal{O}(m))$  can be written in the following form $f\otimes 1$ where $f=\sum_{j=0}^m
f_j z^j$ and $f_j$ are smooth functions on $\mathbf{P}^1$. It follows that  an
 invariant element in $A^{(0,0)}(\mathbf{P}^1,\mathcal{O}(m))$ corresponds to a smooth function $f$ of the previous form invariant under the action of  $\mathbf{S}^1$.  To  this function $f$  we associate a smooth function $\phi_f$ on $]0,1[$ as follows
 $\phi_f(x)=f(\exp(-G(x)))$, for any  $x\in ]0,1[$. We set \[\overline{h}(x):=h(1,1)(\exp (-G_\omega(x))), \quad x\in  ]0,1[,
 \] and we consider the following norm on $\mathcal{C}^\infty
(]0,1[)$
\[
\|\varphi\|^2_2=\int_0^1\overline{h}_\omega(x)\varphi(x)^2dx+\int_0^1 \overline{h}_\omega(x)|\varphi'(x)|^2 \overline{g}_\omega(x)dx, \quad \varphi \in \mathcal{C}^\infty
(]0,1[).
\]
One checks easily that $\|\phi_f\|_2=\|f\otimes 1\|_2$. We set $ K_{\omega,h}=\{\varphi\in \mathcal{C}^\infty
(]0,1[)|\, \|\varphi\|^2<\infty\}$  and we denote by $H_2^0$ the completion of $ K_{\omega,h}$  with respect to $\|\cdot\|_2$.
This completion doesn't depend on the choice of the metrics, since it is the restriction of
$\|\cdot\|_2$ to the space
of invariant elements in $H_2$ which doesn't depend on $\omega$ and on $h$.\\

We denote by $R_{\omega,h}$ the following functional
\[
R_{\omega,h}(\varphi):=\frac{\int_0^1 \overline{h}_\omega(x)\overline{g}_\omega(x)|\varphi'(x)|^2
dx}{\int_0^1\overline{h}_\omega(x)\varphi(x)^2dx},\quad
 \varphi\in K_{\omega,h}\setminus\{0\}.
\]
 We denote by $0=\lambda_0(\omega,h)<\lambda_1(\omega,h)<\lambda_1(\omega,h)\ldots$ the
 invariant eigenvalues of $\Delta_{\omega,h}$ then by the Min-Max principle,
 \[
 \lambda_j(\omega,h)=\inf_{\varphi\in K_{\omega,h,j}\setminus\{0\}}R_{\omega,h}(\varphi), \quad  j=1,2,\ldots
 \]
where  $K_{\omega,h,j}$ is  the orthogonal to the subspace of $H_2^0$ spanned by the eigenfunctions associated to
$\lambda_k(\omega,h)$ for $k=0,\ldots,j-1$.
 \begin{theorem}\label{y4}
 Suppose that $1<\frac{\overline{h}_\omega}{\overline{h}_{m,\infty}}<\frac{\overline{g}_{can}}{\overline{g}_\omega}$. Then
 \[
 \lambda_j(\omega,h)\leq \lambda_j(\omega_{can},h_{m,\infty}),\quad
  j=1,2,\ldots
 \]

 \end{theorem}

 \begin{proof} Suppose that $1<\frac{\overline{h}}{\overline{h}_{m,\infty}}<\frac{\overline{g}_{can}}{\overline{g}_\omega}$. Then,

 \[
 R_{\omega,h}(\varphi)\leq \frac{\int_0^1 \overline{h}_{m,\infty}(x)\overline{g}_{can}(x)|\varphi'(x)|^2 dx}{\int_0^1
 \overline{h}_{m,\infty}(x)\varphi(x)^2dx},\quad
  \varphi\in K_{\omega,h}\setminus\{0\}.
 \]
 Notice that
 \[
 R_{\omega_{can},h_{m,\infty}}(\varphi)= \frac{\int_0^1\overline{h}_{m,\infty}(x) \overline{g}_{can}(x)|\varphi'(x)|^2 dx}{\int_0^1\overline{h}_{m,\infty}(x)\varphi(x)^2
 dx},\quad
 \varphi\in K_{\omega,h}\setminus\{0\}\,\,\,\text{(see the notations of 
 Example \ref{z3}, \eqref{metcan} and \eqref{metcan1})}.
 \]
By the monotonicity principle and Theorem \ref{y2}, we obtain
\[
\lambda_j(\omega,h)\leq \lambda_j(\omega_{can},h_{m,\infty}),\quad j=1,2,\ldots 
\]
 \end{proof}

In particular, when $m=0$ and $h=h_{0,\infty}$ is the constant metric on $\mathcal{O}$,  Theorem  \ref{y4} becomes
\[
\lambda_j(\omega,h)\leq \lambda_j(\omega_{can}, h_{0,\infty})=:\lambda(\omega_{can}),\quad j=1,2,\ldots 
\]
and since $\{\lambda_k(\omega_{can}, h_{0,\infty}),\, k\in \mathbf{N}_{\geq 1}\}=\{\frac{\alpha^2}{2}|J_0(\alpha)J'_0(\alpha)=0\}$, then we recover
 the result of \cite[theorem 2]{Abreu2}.

\subsection{Large first invariant eigenvalue}\label{www1}
In this paragraph, we show that Theorem \ref{y4} does not hold if
$\overline{h}_\omega$ and $\overline{g}_\omega$ do not satisfy the
condition of the theorem. More precisely, we construct 
$(\omega_\mu)_{\mu\geq 0}$ (resp. $(h_{\omega_\mu})_{\mu\geq 0}$) 
a sequence of smooth normalized volume forms (resp. smooth metrics) on $\mathbf{P}^1$
(resp. on $\mathcal{O}(m)$) such that 
\[
\lim_{\mu\rightarrow \infty}\sup_{x\in [0,1]}\frac{\overline{g}_{\omega_\mu}\overline{h}_{\omega_\mu}}{
\overline{g}_{can}\overline{h}_{\infty,m}}=\infty,\quad\text{and}\quad
\lim_{\mu\rightarrow \infty}\lambda_1(\omega_\mu,h_{\omega_{\mu}})=\infty.
\]

The following construction is a slight generalization of the construction
in \cite[Paragraph 3.1]{Abreu2}.
Let $\mu\in \mathbf{R}^+$, we set 
$\overline{g}_\mu(x):=2x(1-x)(1+4\mu x(1-x))$ for any $x\in [0,1]$. Then
$\overline{g}_\mu$ defines a   smooth normalized volume form 
on $\mathbf{P}^1$ denoted
$\omega_\mu$, see the construction of this article or 
\cite[Paragraph 3.1]{Abreu2}.  By Corollary 
\ref{y7}, there exists $h_{\omega_\mu}$,
 a smooth and invariant metric on $\mathcal{O}(m)$ such that 
 $\overline{h}_{\omega_\mu}(x):=h_{\omega_\mu}(1,1)(\exp(-G_{\omega_\mu}(x))=(1-x)^m$ for any
 $x\in [0,1]$. We check easily that $
 \frac{1}{2^m}(1+\mu)\leq \sup_{x\in [0,1]}\frac{\overline{g}_{\omega_\mu}\overline{h}_{\omega_\mu}}{
\overline{g}_{can}\overline{h}_{\infty,m}}$. Thus,
\begin{equation}
\lim_{\mu\rightarrow \infty}\sup_{x\in [0,1]}\frac{\overline{g}_{\omega_\mu}\overline{h}_{\omega_\mu}}{
\overline{g}_{can}\overline{h}_{\infty,m}}=\infty.
\end{equation}

 We have
\[
R_{\omega_\mu,h_{\omega_\mu}}(\varphi)=\frac{2\int_0^1 (1-x)^m
 x(1-x)(1+4\mu x(1-x))|\varphi'(x)|^2
dx}{\int_0^1(1-x)^m|\varphi(x)|^2dx},\quad
 \varphi\in K_{\omega_\mu,h_\mu}\setminus\{0\}.
\]
This gives,
\begin{equation}
R_{\omega_\mu,h_{\omega_\mu}}(\varphi)\geq 8\mu \frac{\int_0^1 (1-x)^m
 x^2(1-x)^2|\varphi'(x)|^2
dx}{\int_0^1(1-x)^m|\varphi(x)|^2dx},\quad
 \varphi\in K_{\omega_{\mu},h_{\omega_\mu}}\setminus\{0\}
\end{equation}

Then,
\begin{equation}
\lambda_1(\omega_\mu,h_{\omega_\mu} )\geq 8\mu\inf_{\varphi \in K_{\omega_\mu,
h_{{}_{\omega_\mu}},1}}  \frac{\int_0^1 (1-x)^m
 x^2(1-x)^2|\varphi'(x)|^2
dx}{\int_0^1(1-x)^m|\varphi(x)|^2dx}.
\end{equation}
  We claim that $\inf_{{}_{\varphi \in K_{{}_{\omega_\mu,
h_{{}_{\omega_\mu}},1}}}}  \frac{2\int_0^1 (1-x)^m
 x^2(1-x)^2|\varphi'(x)|^2
dx}{\int_0^1(1-x)^m|\varphi(x)|^2dx}\neq 0$, this  will imply that
$\lambda_1(\omega_\mu,h_{\omega_\mu} )\rightarrow \infty$ as 
$\mu\rightarrow \infty$. Let us prove that 

\begin{equation}\label{e1}
\inf_{{}_{\varphi \in K_{{}_{\omega_\mu,
h_{{}_{\omega_\mu}},1}}}}  \frac{2\int_0^1 (1-x)^m
 x^2(1-x)^2|\varphi'(x)|^2
dx}{\int_0^1(1-x)^m|\varphi(x)|^2dx}\neq 0.
\end{equation} 
 First, we prove the following 
 
  \begin{lemma}\label{t2}
Let $\psi\in \mathcal{C}^\infty([0,1])$ with $\psi(0)=0$. We have
\[
\int_0^1 (1-x)^m(1-x^2)^2|\psi'(x)|^2dx\geq (m+1)\int_0^1(1-x)^m |\psi(x)|^2dx. 
\] 
\end{lemma}

\begin{proof}
From 
\[
0\leq \int_0^1\Bigl(-\frac{\psi(x)}{x}+(1-x^2)\psi'(x)  \Bigr)^2(1-x)^mdx,
\]
and integrating by parts we obtain
\[
\int_0^1(1-x)^m(1-x^2)^2|\psi'(x)|^2dx\geq \int_0^1(1-x)^m|\psi(x)|^2
\bigl(1+m+\frac{m}{x} \bigr)dx\geq (m+1)\int_0^1(1-x)^m|\psi(x)|^2
dx.
\]
\end{proof}

\begin{corollary}
There exists a constant $M>0$, such that for any  $\varphi$  a smooth function in $ K_{\omega_\mu,h_\mu,1}$,  we have
\begin{equation}\label{t5}
2\int_0^1(1-x)^m x^2(1-x)^2|\varphi'(x)|^2dx \geq M \int_0^1(1-x)^m |
\varphi(x)|^2dx.
\end{equation}

\end{corollary}
\begin{proof} 

Let $\varphi$ be a smooth function in  $K_{\omega_\mu,h_\mu,1}$. We have
$\int_0^1\overline{h}_{\omega_\mu}(x)\varphi(x) dx=\int_0^1(1-x)^m \varphi(x)dx
=0 $. We set 
$\psi(x):=\varphi(\frac{x+1}{2})$ for any $x\in [-1,1]$. Clearly, 
$\int_{-1}^1(1-x)^m \psi(x)dx
=0 $.  We have
\[
\int_{-1}^0(1-x)^m(1-x^2)^2\bigl|
[\psi(x)-\psi(0)]'\bigr|^2dx\geq \int_{-1}^0(1-x^2)^2\bigl|
[\psi(x)-\psi(0)]'\bigr|^2dx\geq C \int_{-1}^0\bigl|
\psi(x)-\psi(0)\bigr|^2dx,
\]
where the last inequality follows from \cite[Lemma 3.1]{Abreu2}. Thus,

\begin{equation}\label{t1}
\int_{-1}^0(1-x)^m(1-x^2)^2\bigl|
[\psi(x)-\psi(0)]'\bigr|^2dx \geq \frac{C}{2^m} 
\int_{-1}^0(1-x)^m\bigl|
\psi(x)-\psi(0)\bigr|^2dx,
\end{equation}
On other hand, we have by Lemma \ref{t2} 
\begin{equation}\label{t3}
\int_0^1(1-x)^m(1-x^2)^2|[\psi(x)-\psi(0)]'|^2dx\geq (m+1) 
\int_0^1(1-x)^m\bigl|
\psi(x)-\psi(0)\bigr|^2dx
\end{equation}
We let $M':=\min(\frac{C}{2^m},m+1)$. From \eqref{t1} and \eqref{t3}, we obtain
\[
\begin{split}
\int_{-1}^1(1-x)^m(1-x^2)^2|\psi'(x)|^2dx &\geq M'
\int_{-1}^1(1-x)^m \bigl|
\psi(x)-\psi(0)\bigr|^2dx\\
= &  M'\int_{-1}^1(1-x)^m |\psi(x)|^2dx
-M' \psi(0)\int_{-1}^1(1-x)^m \psi(x)dx\\
&+
M'\psi(0)^2\int_{-1}^1 (1-x)^mdx\\
\geq & M'\int_{-1}^1(1-x)^m |\psi(x)|^2dx.
\end{split}
\]
Since 
$\psi(x)=\varphi(\frac{x+1}{2})$ for any $x\in [-1,1]$, 
we deduce \eqref{t5} and then \eqref{e1}.

\end{proof}

% Use this code if you wish to generate your bibliography with BibTeX;
% please replace first the string "demo" below with the name(s) of
% the BibTeX data base(s) you want to use.
% The resulting bibliography-output (the contents of the .bbl file)
% must be pasted into this file before submission.
% 
% \bibliographystyle{mn}
% \bibliography{demo}
% 
% Replace the following example bibliography with your references
% before submission:

\bibliographystyle{plain}
\bibliography{biblio}
\end{document}